\documentclass[12pt, a4paper]{article}

\usepackage[latin2]{inputenc}
\usepackage{amssymb}
\usepackage{paralist}
\usepackage{amsmath, calc}
\usepackage{amsfonts}
\usepackage{amsthm}
\usepackage{fullpage}
\usepackage{enumerate}
\usepackage{paralist}

\newtheorem{theorem}{Theorem}[section] %
\newtheorem{lemma}[theorem]{Lemma} %
\newtheorem{corollary}[theorem]{Corollary} %

\newtheorem*{theoremA}{Theorem A}
\newtheorem*{theoremB}{Theorem B}
\newtheorem*{theoremC}{Theorem C}
\newtheorem*{theoremD}{Theorem D}
\newtheorem*{theoremE}{Theorem E}
\newtheorem*{theoremF}{Theorem F}
\usepackage{cite}

\begin{document}

\title{On monotone increasing representation functions}

\author{
S\'andor Z. Kiss \thanks{Department of Algebra, Institute of Mathematics, Budapest University of
Technology and Economics, M\H{u}egyetem rkp. 3., H-1111, Budapest, Hungary. Email: ksandor@math.bme.hu.
This author was supported by the NKFIH Grant No. K129335.}, Csaba
S\'andor \thanks{Department of Stochastics, Institute of Mathematics, Budapest University of
Technology and Economics, M\H{u}egyetem rkp. 3., H-1111, Budapest, Hungary. Department of Computer Science and Information Theory, Budapest University of Technology and Economics, M\H{u}egyetem rkp. 3., H-1111 Budapest, Hungary, MTA-BME Lend\"ulet Arithmetic Combinatorics Research Group,
  ELKH, M\H{u}egyetem rkp. 3., H-1111 Budapest, Hungary . Email: csandor@math.bme.hu.
This author was supported by the NKFIH Grant No. K129335.},
Quan-Hui Yang \thanks{School of Mathematics and Statistics, Nanjing University of Information Science and Technology, Nanjing 210044, China; yangquanhui01@163.com.}
}

\date{}

\maketitle

\begin{abstract}
\noindent Let $k\ge 2$ be an integer and let $A$ be a set of nonnegative integers. The representation function $R_{A,k}(n)$ for the set $A$ is the number of representations of a nonnegative integer $n$ as the sum of $k$ terms from $A$. Let $A(n)$ denote the counting function of $A$.
Bell and Shallit recently gave a counterexample for a conjecture of Dombi and proved that if
$A(n)=o(n^{\frac{k-2}{k}-\epsilon})$ for some $\epsilon>0$, then $R_{\mathbb{N}\setminus A,k}(n)$
is eventually strictly increasing. In this paper, we improve this result to $A(n)=O(n^{\frac{k-2}{k-1}})$. We also give an example to show that this bound is best possible.

 {\it
2010 Mathematics Subject Classification:} 11B34

{\it Keywords and phrases:} additive number theory; additve representation function; monotonicity
\end{abstract}

\section{Introduction}

Let $\mathbb{N}$ be the set of nonnegative integers and let $A$ be a subset of nonnegative integers.
We use $A^n$ to denote the Cartesian product of $n$ sets $A$, that is,
$$A^n=\{(a_1,a_2,\ldots,a_n):~a_1,a_2,\ldots,a_n\in A\}.$$
Let $$R_{A,k}(n)=|\{ (a_{1}, a_2,\dots{}, a_{k})\in A^n: a_{1} +a_2+\cdots + a_{k} = n  \}|,
$$
$$R^{<}_{A,k}(n)=|\{ (a_{1}, a_2,\dots{}, a_{k})\in A^n: a_{1} + a_2+\cdots+ a_{k} = n, a_{1} <a_2< \cdots{} < a_{k}  \}|,
$$
$$R^{\le}_{A,k}(n)=|\{ (a_{1},a_2, \dots{} a_{k})\in A^n:a_{1} +a_2+ \cdots + a_{k} = n, a_{1} \le a_2\le \cdots{} \le a_{k}  \}|,
$$
where $|.|$ denotes the cardinality of a finite set. We say that $R_{A,k}(n)$ is monotonous increasing in $n$ from a certain point on (or eventually monotone increasing), if there exists an integer $n_{0}$ such that $R_{A,k}(n+1) \ge R_{A,k}(n)$ for all integers $n\ge n_0$.

We define the monotonicity of the two other representation functions $R^{<}_{A,k}(n)$ and $R^{\le}_{A,k}(n)$ in the same way. We denote the counting function of the set $A$ by
\[
A(n) = \sum_{\overset{a \in \mathcal{A}}{a \le n}}1.
\]
We define the lower asymptotic density of a set $A$ of natural numbers by
\[
\liminf_{n \rightarrow \infty}\frac{A(n)}{n}
\]
and the asymptotic density by
\[
\lim_{n \rightarrow \infty}\frac{A(n)}{n}
\]
whenever the limit exists. The generating function of a set $A$ of natural numbers is denoted by
\[
G_{A}(x) = \sum_{a\in A}x^{a}.
\]

Obviously, if $\mathbb{N}\setminus A$ is finite, then each of the functions $R_{A,2}(n), R^{<}_{A,2}(n)$ and
$R^{\le}_{A,2}(n)$ is eventually monotone increasing. In \cite{EST}, \cite{ESV}, Erd\H{o}s, S\'ark\"ozy and V.T. S\'os investigated whether there exists a set $A$ for which $\mathbb{N}\setminus A$ is infinite and the representation functions are monotone increasing from a certain point on. Namely, they proved the following theorems.

\begin{theoremA}
The function $R_{A,2}(n)$ is monotonous increasing from a certain point on, if and only if the sequence $A$ contains all the integers
from a certain point on, i.e., there exists an integer $n_{1}$ with
\[
A \cap \{n_{1}, n_{1}+1, n_{1}+2, \dots{} \} =  \{n_{1}, n_{1}+1, n_{1}+2, \dots{} \}.
\]
\end{theoremA}

\begin{theoremB}
There exists an infinite set $A\subseteq \mathbb{N}$ such that
$A(n) < n - cn^{1/3}$ for $n > n_{0}$ and $R^{<}_{A,2}(n)$ is monotone increasing from a certain point on.
\end{theoremB}

\begin{theoremC}
If
$$A(n) = o\Big(\frac{n}{\log n}\Big),$$
then the functions $R^{<}_{A,2}(n)$ and
$R^{\le}_{A,2}(n)$ cannot be monotonous increasing in $n$ from a certain point on.
\end{theoremC}

\begin{theoremD}
If $A\subseteq \mathbb{N}$ is an infinite set with
$$\lim_{n\rightarrow \infty}\frac{n-A(n)}{\log n} = \infty,
$$
then $R^{\le}_{A,2}(n)$ cannot be monotone increasing from a certain point on.
\end{theoremD}

The latter theorem was proved by Balasubramanian \cite{Bala} independently.

Very little is known when $k > 2$. Many years ago, it was proved in \cite{tang} and independently in \cite{Kiss} that

\begin{theoremE}
If $k$ is an integer with $k > 2$, $A \subseteq \mathbb{N}$ and
$R_{A,k}(n)$ is monotonous increasing in $n$ from a certain point on,
then
\[
A(n) = o\bigg(\frac{n^{2/k}}{(\log n)^{2/k}}\bigg)
\]
cannot hold.
\end{theoremE}

Furthermore, Dombi \cite{Dom} constructed sets $A$ of asymptotic density $\frac{1}{2}$ such that for $k > 4$, the function $R_{A,k}(n)$ is monotone increasing from a certain point on. His constructions are based on the Rudin-Shapiro sets and Thue-Morse sequences. On the other hand, Dombi gave the following
conjecture.

\noindent{\bf Dombi's Conjecture.} If $\mathbb{N}\setminus A$ is infinite, then $R_{A,k}(n)$ cannot be strictly increasing.

For $k \ge 3$, Bell and Shallit recently gave a counterexample \cite{Bell} of Dombi's conjecture by applying tools from automata theory and logic. They proved the following result.

\begin{theoremF} Let $k$ be integer with $k\ge 3$ and let $F\subseteq \mathbb{N}$ with $0\notin F$. If $F(n) = o(n^{\alpha})$ for $\alpha < (k-2)/k$ and $A = \mathbb{N}\setminus F$, then $R_{A,k}(n)$ is eventually strictly increasing.
\end{theoremF}
In this paper, we improve their result in the following theorem.

\begin{theorem} Let $k$ be an integer with $k\ge 3$. If $A \subseteq \mathbb{N}$ satisfies
$$A(n) \leqslant \frac{n^{\frac{k-2}{k-1}}}{\sqrt[k-1]{(k-2)!}} - 2$$
for all sufficiently large integers $n$, then $R_{\mathbb{N}\setminus A,k}(n)$ is eventually strictly increasing.
\end{theorem}
In particular, for $k = 3$ we have the following Corollary.
\begin{corollary}
If $A \subseteq \mathbb{N}$ satisfies $A(n) \leqslant \sqrt{n} - 2$
for all sufficiently large integers $n$, then $R_{\mathbb{N}\setminus A,3}(n)$ is eventually strictly increasing.
\end{corollary}

We improve the constant factor above in the following theorem.

\begin{theorem} Let $A \subseteq \mathbb{N}$ satisfies $A(n) \leqslant \frac{2}{\sqrt{3}} \sqrt{n}-2$ for all sufficiently large integers $n$, then $R_{\mathbb{N}\setminus A,3}(n)$ is eventually strictly increasing.
\end{theorem}

It turns out from the next theorem that the upper bound for the counting function of $A$ in Theorem 1.1 is tight up to a constant factor.

\begin{theorem} Suppose that $f(n)$ be a function satisfying $f(n)\rightarrow \infty$
as $n\rightarrow \infty$. Then
there is a set $A\subseteq \mathbb{N}$ such that $A(n)<\sqrt[k-1]{k-1}n^{\frac{k-2}{k-1}}+f(n)$ for
all sufficiently large integers $n$ and $R_{\mathbb{N}\setminus A,k}(n)<R_{\mathbb{N}\setminus A,k}(n-1)$
for infinitely many positive integers $n$.
\end{theorem}

Note that Shallit \cite{Sha} recently constructed a set $A$ with positive lower asymptotic density such that the function $R_{\mathbb{N}\setminus A,3}(n)$
is strictly increasing.

\section{Proofs}

The proofs of theorems are based on the next lemma, coming from Bell and Shallit's paper \cite{Bell} yet not explicitly stated there.  

\begin{lemma} For any positive integers $n$ and $k$ with $k\ge 3$, we have
\begin{eqnarray*}&&R_{\mathbb{N}\setminus A,k}(n)-R_{\mathbb{N}\setminus A,k}(n-1)\\
&=&\binom{n+k-2}{k-2}+\sum_{i=1}^{k-2}\binom{k}{i}(-1)^{i}\left(\sum_{m=0}^{n}\binom{m+k-i-2}{k-i-2}R_{A,i}(n-m)\right)\\
&&+ (-1)^{k-1}kR_{A,k-1}(n)+(-1)^{k}(R_{A,k}(n)-R_{A,k}(n-1)).
\end{eqnarray*}
\end{lemma}

\begin{proof}[\bf Proof of Lemma 2.1.] Clearly,
\begin{eqnarray*}
(1-x)(G_{\mathbb{N}\setminus A}(x))^{k}& = &\sum_{n=0}^{\infty}R_{\mathbb{N}\setminus A,k}(n)x^{n}-\sum_{n=0}^{\infty}R_{\mathbb{N}\setminus A,k}(n)x^{n+1}\\
&=&R_{\mathbb{N}\setminus A,k}(0)+\sum_{n=1}^{\infty}(R_{\mathbb{N}\setminus A,k}(n)-R_{\mathbb{N}\setminus A,k}(n-1))x^{n}.
\end{eqnarray*}
On the other hand,
\begin{eqnarray*}
&&(1-x)((G_{\mathbb{N}\setminus A)}(x))^{k} = (1-x)\left(\frac{1}{1-x}-G_{A}(x)\right)^{k}
= (1-x)\sum_{i=0}^{k}\binom{k}{i}\frac{(-1)^{i}}{(1-x)^{k-i}}G_{A}(x)^{i}\\
&=&\frac{1}{(1-x)^{k-1}} + \sum_{i=1}^{k-2}\binom{k}{i}\frac{(-1)^{i}}{(1-x)^{k-i-1}}G_{A}(x)^{i}+(-1)^{k-1}kG_{A}(x)^{k-1}+(-1)^{k}(1-x)G_{A}(x)^{k}.
\end{eqnarray*}
It is well known that
\[
\frac{1}{(1-x)^{m}} = \sum_{n=0}^{\infty}\binom{n+m-1}{m-1}x^{n}.
\]
It follows that
\begin{eqnarray*}
&&R_{\mathbb{N}\setminus A,k}(0)+\sum_{n=1}^{\infty}(R_{\mathbb{N}\setminus A,k}(n)-R_{\mathbb{N}\setminus A,k}(n-1))x^{n}\\
&=&\sum_{n=0}^{\infty}\binom{n+k-2}{k-2}x^{n} + \sum_{i=1}^{k-2}(-1)^{i}\binom{k}{i}\sum_{n=0}^{\infty}\left(\sum_{m=0}^{n}\binom{m+k-i-2}{k-i-2}R_{A,i}(n-m)\right)x^{n}\\
&&+(-1)^{k-1}k\sum_{n=0}^{\infty}R_{A,k-1}(n)x^{n}+(-1)^{k}R_{A,k}(0)+(-1)^{k}\sum_{n=0}^{\infty}(R_{A,k}(n)-R_{A,k}(n-1))x^{n}.
\end{eqnarray*}
Comparing the coefficient of $x^{n}$ on both sides of the equation, Lemma 2.1 follows immediately.
\end{proof}

\begin{proof}[\bf Proof of Theorem 1.1.] Clearly,
\begin{eqnarray*}
&&R_{A,i}(n)=|\{ (a_{1}, a_2,\dots{}, a_{i})\in A^i:  a_{1} + a_2+\cdots{} + a_{i} = n  \}|\\
&&\le |\{ (a_{1},a_2, \dots{},a_{i-1})\in A^{i-1}: a_{1},a_2, \dots{} ,a_{i-1} \le n \}| = A(n)^{i-1}.
\end{eqnarray*}
Then by Lemma 1, there exist constants $c_{1},c_{2},c_{3}, c_{4}$ only depending on $k$ such that
\begin{eqnarray*}
&&R_{\mathbb{N}\setminus A,k}(n)-R_{\mathbb{N}\setminus A,k}(n-1)\\
&=&\binom{n+k-2}{k-2}+\sum_{i=1}^{k-2}\binom{k}{i}(-1)^{i}\left(\sum_{m=0}^{n}\binom{m+k-i-2}{k-i-2}R_{A,i}(n-m)\right)\\
&&+(-1)^{k-1}kR_{A,k-1}(n)+(-1)^{k}(R_{A,k}(n)-R_{A,k}(n-1))\\
&\ge&\frac{n^{k-2}}{(k-2)!}-\sum_{i=1}^{k-2}2^{k}\sum_{m=0}^{n}\binom{m+k-i-2}{k-i-2}A(n)^{i-1}-kR_{A,k-1}(n)-A(n)^{k-1}\\
&\ge&\frac{n^{k-2}}{(k-2)!}-\sum_{i=1}^{k-2}2^{k}A(n)^{i-1}\binom{n+k-i-1}{k-i-2}
-k\bigg(\frac{n^{\frac{k-2}{k-1}}}{\sqrt[k-1]{(k-2)!}}\bigg)^{k-2}
-\bigg(\frac{n^{\frac{k-2}{k-1}}}{\sqrt[k-1]{(k-2)!}}-2\bigg)^{k-1}\\
&\ge&\frac{n^{k-2}}{(k-2)!}-c_{1}\sum_{i=1}^{k-2}A(n)^{i-1}n^{k-i-2}-k\cdot\frac{n^{\frac{(k-2)^{2}}{k-1}}}{((k-2)!)^{\frac{k-2}{k-1}}}\\
&&-\bigg(\frac{n^{k-2}}{(k-2)!}-2(k-1)\frac{n^{\frac{(k-2)^2}{k-1}}}
{((k-2)!)^{\frac{k-2}{k-1}}}+c_{2}n^{\frac{(k-2)(k-3)}{k-1}}\bigg)\\
&\ge&\frac{n^{k-2}}{(k-2)!}-c_{3}n^{k-3}-k\frac{n^{\frac{(k-2)^{2}}{k-1}}}{((k-2)!)^{\frac{k-2}{k-1}}}
-\bigg(\frac{n^{k-2}}{(k-2)!}-\frac{2(k-1)n^{\frac{(k-2)^{2}}{k-1}}}{((k-2)!)^{\frac{k-2}{k-1}}}
+c_{2}n^{\frac{(k-2)(k-3)}{k-1}}\bigg)\\
&=&\frac{k-2}{((k-2)!)^{\frac{k-2}{k-1}}}\cdot n^{\frac{(k-2)^{2}}{k-1}}-c_{4}n^{k-3}.
\end{eqnarray*}
Hence $R_{\mathbb{N}\setminus A,k}(n)-R_{\mathbb{N}\setminus A,k}(n-1)>0$ when $n$ is large enough.
\end{proof}

\begin{lemma} For any set $A$ of natural numbers and for any natural number $n$, one has $R_{A, 3}(n) \leqslant \frac{3}{4} A(n)^2+\left\{\frac{A(n)^{2}}{4}\right\},$ where $\{x\}$ denotes the
fractional part of $x$.
\end{lemma}

Note that Lemma 2.2 is sharp: if $A = \{0,1,\dots{} ,m\}$, then
\[
R_{A, 3}\left(\left \lfloor \frac{3m}{2} \right \rfloor \right) = \frac{3}{4} A\left(\left \lfloor \frac{3m}{2} \right \rfloor\right)^2+\left\{\frac{A\left(\left \lfloor \frac{3m}{2} \right \rfloor\right)^{2}}{4}\right\},
\]
where $\lfloor y\rfloor$ denotes the maximal integer not greater than $y$.

\begin{proof}[\bf Proof of Lemma 2.2.] Fix a natural number $n$. Let $A \cap[1, n]=\left\{a_1<a_2<\cdots<a_m\right\}$ and $\overline{A}=\{n-a_m<n-a_{m-1}<\cdots<n-a_1\}$.
For $i=1,2,\ldots,m$, we define
$$A_i=\{a_i+a_1<a_i+a_2<\cdots<a_i+a_{m+1-i}<a_{i+1}+a_{m+1-i}<\cdots<a_m+a_{m+1-i}\}.$$
Clearly,

\begin{eqnarray*}
R_{A, 3}(n)&=&\sum_{i=1}^m\left|A_i\cap \overline{A}\right| \leqslant \sum_{i=1}^m \min \{2 m-2 i+1, m\}\\
&=&\sum_{i=1}^{\left\lfloor\frac{m}{2}\right\rfloor} m+\sum_{i=\left\lfloor\frac{m}{2}\right\rfloor+1}^m(2 m-2 i+1)\\
&=&m\left\lfloor\frac{m}{2}\right\rfloor+\left(m-\left\lfloor\frac{m}{2}\right\rfloor\right)^2=\frac{3}{4} m^2+\left\{\frac{m^2}{4}\right\} .
\end{eqnarray*}
\end{proof}

\begin{proof}[\bf Proof of Theorem 1.3.]
Applying Lemma 2.1 for $k = 3$, we have
\[
R_{\mathbb{N} \backslash A,3}(n)-R_{\mathbb{N} \backslash A,3}(n-1) = n + 1 -3\sum_{m=0}^{n}R_{A,1}(n-m)+3R_{A,2}(n)-(R_{A,3}(n)-R_{A,3}(n-1))
\]
\[
= n + 1 - 3A(n) + 3R_{A,2}(n) - (R_{A, 3}(n) - R_{A, 3}(n-1)).
\]
Hence by Lemma 2.2,
\[
R_{\mathbb{N} \backslash A,3}(n)-R_{\mathbb{N} \backslash A,3}(n-1)\ge n + 1 - 3A(n) - R_{A,3}(n)
\]
\[
\ge n+1-3\left(\frac{2}{\sqrt{3}}\sqrt{n}-2\right)
-\frac{3}{4}\left(\frac{2}{\sqrt{3}}\sqrt{n}-2\right)^{2}-\frac{1}{4}=\frac{15}{4}> 0,
\]
which completes the proof.
\end{proof}

\bigskip

\begin{proof}[\bf Proof of Theorem 1.4.]
We may suppose that $f(n)<\sqrt[k-1]{k-1}n^{\frac{k-2}{k-1}}$.
We define an infinite sequence of natural numbers $N_{1}, N_{2}, \dots{}$ by induction.
Let $N_{1} = 100k^4$. Assume that $N_{1}, \dots{} ,N_{j}$ is already defined. Let $N_{j+1}$ be an even number with $N_{j+1} > 100k^4N_{j}^{k-1}$ and $f(n) > (k-1)(N_{1}^{k-2} + \cdots{} + N_{j}^{k-2})$ for every $n \ge N_{j+1}$. We define the set $A$ in the following way.
\[
A = \bigcup_{j=1}^{\infty}\{N_{j}, 2N_{j}, 3N_{j}, \dots{} ,(k-1)N_{j}^{k-1}\}.
\]

First, we give an upper estimation for $A(n)$. Let $n\ge 100k^4$. Then there exists an index $j$ such that
$N_{j} \le n < N_{j+1}$. Define $l$ as the largest integer with $l \le (k-1)N_{j}^{k-2}$ and $lN_{j} \le n$. Then we have
\begin{eqnarray*}&&A(n)-\sqrt[k-1]{k-1}n^{\frac{k-2}{k-1}}\\
 &\le & (k-1)(N_{1}^{k-2} + \cdots{} + N_{j}^{k-2}) + l - \sqrt[k-1]{k-1}(lN_{j})^{\frac{k-2}{k-1}}\\
&=&(k-1)(N_{1}^{k-2} + \cdots{} + N_{j}^{k-2}) + l^{\frac{k-2}{k-1}}(l^{\frac{1}{k-1}}-(k-1)^{\frac{1}{k-1}}N_{j}^{\frac{k-2}{k-1}})\\
&\le & f(n),
\end{eqnarray*}
which implies that
$$
A(n)<\sqrt[k-1]{k-1}n^{\frac{k-2}{k-1}}+f(n).
$$

Next we shall prove that there exist infinitely many positive integers $n$ such that
$R_{\mathbb{N}\setminus A,k}(n)<R_{\mathbb{N}\setminus A,k}(n-1)$. In order to prove this,
we divide into two cases according to the parity of $k$.

Suppose that $k$ is an odd integer. For $j=1,2,\ldots$, we define
$$
u_j=(k-1)N_{j}^{k-1} + 100(k-2)(k-1)^3N_{j}^{k-2}.
$$
Now, we show that $R_{\mathbb{N}\setminus A,k}(u_j) < R_{\mathbb{N}\setminus A,k}(u_j-1)$ when $j$ is large enough.

Since all the elements of $A$ are even and $u_j-1$ is odd, it follows that $R_{A,k}(u_j-1)=0$.
By Lemma 2.1, we have
\begin{eqnarray}
&&\nonumber R_{\mathbb{N}\setminus A,k}(u_j) - R_{\mathbb{N}\setminus A,k}(u_j-1)\\
&=&\nonumber \binom{u_j+k-2}{k-2}+\sum_{i=1}^{k-2}\binom{k}{i}(-1)^{i}\left(\sum_{m=0}^{u_j}\binom{m+k-i-2}{k-i-2}R_{A,i}(u_j-m)\right)\\
&&\nonumber+ (-1)^{k-1}kR_{A,k-1}(u_j)+ (-1)^{k}(R_{A,k}(u_j)-R_{A,k}(u_j-1))\\
&\le&\nonumber\binom{u_j+k-2}{k-2}+k^2\left(\sum_{m=0}^{u_j}\binom{m+k-4}{k-4}R_{A,2}(u_j-m)\right)\\
&&+\sum_{i=3}^{k-2}2^{k}\sum_{m=0}^{u_j}\binom{m+k-i-2}{k-i-2}A(u_j)^{i-1}+ kR_{A,k-1}(u_j)-R_{A,k}(u_j).
\end{eqnarray}

Next we shall give a bound for each term of the right hand side of (1) respectively.

There exists an constant $c_5$ only depending on $k$ such that
\begin{equation}
\binom{u_j+k-2}{k-2} \le \frac{(k-1)^{k-2}N_{j}^{k^{2}-3k+2}+100(k-2)^2(k-1)^{k}N_{j}^{k^{2}-3k+1}+c_{5}N_{j}^{k^{2}-3k}}{(k-2)!}
\end{equation}
and
\begin{eqnarray}
&&\nonumber k^2\sum_{m=0}^{u_j}\binom{m+k-4}{k-4}R_{A,2}(u_j-m)\le k^2\sum_{m=0}^{u_j}\binom{m+k-4}{k-4}A(u_j-m) \\
&\le& \nonumber k^2\sum_{m=0}^{u_j}\binom{m+k-4}{k-4}A(kN_{j}^{k-1})\le k^2\sum_{m=0}^{u_j}\binom{m+k-4}{k-4}2\sqrt[k-1]{k-1}(kN_j^{k-1})^{\frac{k-2}{k-1}} \\ &\le& \nonumber k^2\sum_{m=0}^{u_j}\binom{m+k-4}{k-4}2k N_j^{k-2}=2k^3N_j^{k-2}\binom{u_j+k-3}{k-3}\\
&\le & 2k^3N_j^{k-2}\binom{kN_j^{k-1}}{k-3}\le \frac{2k^k}{(k-3)!}N_j^{k^2-3k+1}.
\end{eqnarray}
Furthermore,
\begin{eqnarray}
&&\nonumber\sum_{i=3}^{k-2}2^{k}\sum_{m=0}^{u_j}\binom{m+k-i-2}{k-i-2}A(u_j)^{i-1}\\
&\le & \nonumber c_{6}\sum_{i=3}^{k-2}2^{k}\sum_{m=0}^{u_j}\binom{m+k-i-2}{k-i-2}((N_{j}^{k-1})^{\frac{k-2}{k-1}})^{i-1}\\
&\le&\nonumber c_{6}\sum_{i=3}^{k-2}2^{k}N_{j}^{(k-2)(i-1)}\sum_{m=0}^{u_j}\binom{m+k-i-2}{k-i-2}\\
&=& \nonumber c_{6}\sum_{i=3}^{k-2}2^{k}N_{j}^{(k-2)(i-1)}\binom{u_j+k-i-1}{k-i-1}\\
&\le &\nonumber c_{7}\sum_{i=3}^{k-2}N_{j}^{(k-2)(i-1)}\cdot N_{j}^{(k-1)(k-i-1)}\\
&=&c_{7}\sum_{i=3}^{k-2}N_{j}^{k^{2}-3k-i+3} \le c_{8}N_{i}^{k^{2}-3k},
\end{eqnarray}
where $c_{6}, c_{7}$ and $c_{8}$ are constants only depending on $k$.
Moreover,
\begin{eqnarray}
R_{A,k-1}(u_j))&\le & \nonumber A(u_j)^{k-2}
\le A(kN_{j}^{k-1})^{k-2}\\
& \le &(2\sqrt[k-1]{k-1}(kN_{j}^{k-1})^{\frac{k-2}{k-1}})^{k-2}
\le (2k)^{k-2}N_{j}^{(k-2)^{2}}.
\end{eqnarray}
Obviously,
\begin{eqnarray*}
    &&R_{A,k}(u_j) \\
    &\ge&|\{(x_{1}, \dots{} ,x_{k})\in (\mathbb{Z}^+)^k: \sum_{t=1}^k x_t=u_j, N_{j}\mid x_{t}, x_{t} \le (k-1)N_{j}^{k-1}~\text{for}~t=1,\ldots,k\}|\\
    &=&|\{(y_{1}, \dots{} ,y_{k})\in (\mathbb{Z}^+)^k: \sum_{t=1}^k y_t = \frac{u_j}{N_j}, y_{t} \le (k-1)N_{j}^{k-2}~\text{for}~t=1,\ldots,k\}|\\
    &=&|\{(y_{1}, \dots{} ,y_{k})\in (\mathbb{Z}^+)^k: \sum_{t=1}^k y_t  = \frac{u_j}{N_j}\} |\\
    &&-|\{(y_{1}, \dots{} ,y_{k})\in (\mathbb{Z}^+)^k: \sum_{t=1}^k y_t  = \frac{u_j}{N_j},~y_{t} > (k-1)N_{j}^{k-2}~\text{for some}~t\in \{1,\ldots,k\}\}|.
\end{eqnarray*}
We see that
\begin{eqnarray*}&&
|\{(y_{1}, \dots{} ,y_{k})\in (\mathbb{Z}^+)^k: y_{1} + \cdots{} + y_{k} = \frac{u_j}{N_j}\}|\\
&= &\binom{\frac{u_j}{N_j}-1}{k-1}
\ge \frac{N_{j}^{k^{2}-3k+2}+100(k-2)(k-1)^{k+2}N_{j}^{k^{2}-3k+1}+c_{9}N_{j}^{k^{2}-3k}}{(k-1)!},
\end{eqnarray*}
where $c_{9}$ is a constant only depending on $k$, and
\begin{eqnarray*}&&
|\{(y_{1}, \dots{} ,y_{k})\in (\mathbb{Z}^+)^k: \sum_{t=1}^k y_t  = \frac{u_j}{N_j},~y_{t} > (k-1)N_{j}^{k-2}~\text{for some}~t\in \{1,\ldots,k\}\}|\\
&=& k|\{(z_{1}, \dots{} ,z_{k})\in (\mathbb{Z}^+)^k: z_{1} + \cdots{} + z_{k} = 100(k-2)(k-1)^3N_{j}^{k-3}\}|\\
&\le & k(100(k-2)(k-1)^3)^kN_{j}^{k^{2}-3k}.
\end{eqnarray*}
The last equality holds because if $y_{1} + \cdots{} + y_{k} = \frac{u_j}{N_j}$ with $y_{t} > (k-1)N_{j}^{k-1}$,
then $$y_{1} + \cdots{} + y_{t-1} + (y_{t}-(k-1)N_{j}^{k-2}) + y_{t+1} + \cdots{} + y_{k} = 100(k-2)(k-1)^3N_{j}^{k-3},$$
where every terms are positive. Furthermore, if $z_{1} + \cdots{} + z_{k} = 100(k-2)(k-1)^3N_{j}^{k-3}$, $z_{i}\in \mathbb{Z}^{+}$, then one can create $k$ different sums of the form $ y_{1} + \cdots{} + y_{k} = \frac{u_j}{N_j}$ with $y_{i} = z_{i}$ if $i\neq t$ and $y_{t} = z_{t} + (k-1)N_{j}^{k-2}$.

Therefore
\begin{eqnarray}R_{A,k}(u_j) \ge \frac{(k-1)^{k-1}N_{j}^{k^{2}-3k+2}+100(k-2)(k-1)^{k+2}N_{j}^{k^{2}-3k+1}+c_{10}N_{j}^{k^{2}-3k}}{(k-1)!},
\end{eqnarray}
where $c_{10}$ is a constant. In view of (1)$-$(6), we get
\begin{eqnarray*}
&&R_{\mathbb{N}\setminus A,k}(u_j) - R_{\mathbb{N}\setminus A,k}(u_j-1)\\
& \le &\frac{(k-1)^{k-2}N_{j}^{k^{2}-3k+2}+100(k-2)^2(k-1)^k N_{j}^{k^{2}-3k+1}+c_{5}N_{j}^{k^{2}-3k}}{(k-2)!}\\
&& +\frac{2k^k}{(k-3)!}N_j^{k^2-3k+1} + c_{8}N_{i}^{k^{2}-3k} + (2k)^{k-2}N_{j}^{(k-2)^{2}}\\
&&-\frac{(k-1)^{k-1}N_{j}^{k^{2}-3k+2}+100(k-2)(k-1)^{k+2}N_{j}^{k^{2}-3k+1}+c_{10}N_{j}^{k^{2}-3k}}{(k-1)!}\\
&=&\left( \frac{2k^k}{(k-3)!} -100\frac{(k-1)^{k}}{(k-3)!} \right)N_{j}^{k^{2}-3k+1} + (2k)^{k-2}N_{j}^{(k-2)^{2}} + c_{11}N_{j}^{k^{2}-3k}.
\end{eqnarray*}
where $c_{11}$ is a constant. Thus we have $R_{\mathbb{N}\setminus A,k}(u_j) < R_{\mathbb{N}\setminus A,k}(u_j-1)$ when $j$ is large enough.

If $k$ is even, then the same argument as above shows that $R_{\mathbb{N}\setminus A,k}(u_j+1) < R_{\mathbb{N}\setminus A,k}(u_j)$ when $j$ is large enough.
\end{proof}

\end{document}